\documentclass[12pt,reqno]{amsart}
\usepackage{amsfonts}
\usepackage{amsmath,amscd}
\usepackage{fullpage}
\usepackage{amssymb}
\usepackage{centernot} 
\usepackage{enumerate} 
\usepackage{parskip}
\usepackage{pb-diagram} 
\usepackage{mathrsfs}
\usepackage[OT2,T1]{fontenc}
\usepackage{seqsplit}
\usepackage{enumitem}
\usepackage{color}
\usepackage{array}
\usepackage{verbatim}
\usepackage{url}

\DeclareSymbolFont{cyrletters}{OT2}{wncyr}{m}{n}
\DeclareMathSymbol{\Sha}{\mathalpha}{cyrletters}{"58}

\definecolor{refkey}{rgb}{1,1,1}
\definecolor{labelkey}{rgb}{1,1,1}
\definecolor{cite}{rgb}{0.9451,0.2706,0.4941}
\definecolor{ruri}{rgb}{0.0078,0.4022,0.8010}

\usepackage[%
bookmarks=true,bookmarksnumbered=true,%
colorlinks=true,linkcolor=ruri,citecolor=red%
 ]{hyperref}

\makeindex 

\def\F{{\rm \mathbb{F}}}
\def\Z{{\rm \mathbb{Z}}}

\def\Q{{\rm \mathbb{Q}}}

\def\R{{\rm \mathbb{R}}}

\def\P{{\rm \mathbb{P}}}

\def\p{{\rm \mathfrak{p}}}

\def\a{{\rm \mathfrak{a}}}

\def\O{{\rm \mathcal{O}}}

\def\Nm{{\rm Nm}}

\def\avg{{\rm avg}}

\def\Aut{{\rm Aut}}

\def\Cl{{\rm Cl}}

\def\Disc{{\rm Disc}}

\def\SL{{\rm SL}}

\def\Sym{{\rm Sym}}

\def\GL{{\rm GL}}
\def\Gal{{\rm Gal}}

\def\rk{{\rm rk}}

\def\Sel{{\rm Sel}}

\newcommand{\minus}{\scalebox{0.70}[1.0]{$-$}}

\newcommand{\rank}{\mathop{\mathrm{rk}}}

\numberwithin{equation}{section}

\newtheorem{theorem}{Theorem}
\newtheorem{lemma}[theorem]{Lemma}

\newtheorem{remark}[theorem]{Remark}
\newtheorem{definition}[theorem]{Definition}

\newtheorem{corollary}[theorem]{Corollary}
\newtheorem{proposition}[theorem]{Proposition}

\usepackage{marginnote,color}
\usepackage[usenames,dvipsnames]{xcolor}

\usepackage{tikz}

\def\shownotes{\def\inline##1##2##3{ \begin{adjustwidth}{3mm}{7mm}\mbox{}\par \noindent
{\color{##1}\hspace{-1.9cm}{\large ##2}\vspace{-\baselineskip}\\##3}
\newline\end{adjustwidth}} \def\inlinewide##1##2##3{ \begin{adjustwidth}{0mm}{0cm}\mbox{}\par \noindent
{\color{##1}\hspace{-1.6cm}{\large ##2}\vspace{-\baselineskip}\\##3}
\newline\end{adjustwidth}}  \def\marg##1##2##3{\marginnote{\color{##1}{\large ##2}\\{\small ##3}}[-.8cm]}}

\shownotes

\makeatletter
\let\@@pmod\pmod
\DeclareRobustCommand{\pmod}{\@ifstar\@pmods\@@pmod}
\def\@pmods#1{\mkern4mu({\operator@font mod}\mkern 6mu#1)}
\makeatother

\begin{document}
\setlength{\parskip}{2pt}
\setlength{\parindent}{22.5pt}

\title{A positive proportion of cubic fields are not monogenic yet have no local obstruction to being so}

\author{Levent Alp\"oge}
\author{Manjul Bhargava}
\author{Ari Shnidman}

\begin{abstract}
We show that a positive proportion of cubic fields are not monogenic, despite having no local obstruction to being monogenic. Our proof involves the comparison of $2$-descent and $3$-descent in a certain family of Mordell curves $E_k \colon y^2 = x^3 + k$.
As a by-product of our methods, we show that, for every $r \geq 0$, a positive proportion of curves $E_k$
have Tate--Shafarevich group with $3$-rank at least~$r$. 
\end{abstract}

\maketitle

\vspace{-.25in}

\section{Introduction}
One of the most common errors made in a first course in algebraic number
theory is the assumption that every number field $K$ is {\it monogenic}, i.e., that the ring of integers $\O_K$ of $K$ is of the form $\Z[\alpha]$ for some $\alpha\in K$.  While this assertion is true for quadratic fields, it is
expected to be false for 100\% of number fields of any given degree
$d\geq 3$, although this expectation has not been proven for any $d$.

It is easy to construct number fields that fail to be monogenic for {\it local}
reasons.  The first example of such a non-monogenic number field was given by Dedekind, who 
showed that the field $\Q[x]/(x^3 - x^2 - 2x - 8)$, in which 2 splits completely, is not monogenic~\cite{Dedekind}. His result is a special case of the assertion that: if 2 splits completely in a number field $K$ of degree $d\geq 3$, then $K$ cannot be monogenic.  Indeed, if $\O_K\cong \Z[x]/(f(x))$ for a monic integral polynomial $f(x)$ of degree $d\geq 3$, then 2 factors in $\O_K$ as $f(x)$ factors (mod $2$); but $f$ cannot split completely (mod $2$), as there are only two monic linear polynomials (mod $2$). 

However, it has not previously been known even that a positive
proportion of number fields of degree $d\geq3$, {\it that have no local
obstruction to being monogenic}, are not monogenic.  The purpose of this
paper is to prove this result in the case $d=3$, i.e., for cubic~fields.

\begin{theorem}\label{S3main}
When isomorphism classes of cubic fields are ordered by absolute discriminant, a positive proportion are not monogenic and yet have no local obstruction to being
monogenic.
\end{theorem}
\noindent In other words, we prove that a positive proportion of cubic fields, up to isomorphism, are non-monogenic for truly global reasons. In fact we produce such positive proportions of cubic fields having both possible real signatures.

 It is interesting to note that the condition of having {\it no local
 obstruction to being monogenic} is stronger than the condition of being
 {\it locally monogenic}!
 \begin{definition}{\em 
 A number field $K$ is {\it locally monogenic at $p$} if $\O_K \otimes_\Z \Z_p$ is generated by one element as a $\Z_p$-algebra; it is called {\it locally monogenic} if it is locally monogenic at all $p$.}
 \end{definition}
\noindent However, it is possible for a number field to {\it have a local obstruction to being
monogenic} even if it is locally monogenic! 
The obstruction comes from the {\it index
form} $f_K:\O_K/\Z \to \wedge^{d}\O_K$ given by $\alpha\mapsto 1\wedge\alpha\wedge\alpha^2\wedge\cdots\wedge\alpha^{d-1}$. 
A choice of $\Z$-basis for $\O_K$ induces an isomorphism $\wedge^{d}\O_K \simeq \Z$, 
and, via this isomorphism, $f_K$ may be viewed as a homogeneous form of degree~$d \choose 2$ in $d-1$ variables.
Then the number field $K$ is monogenic if and only if $f_K$ represents~$\pm1$ over $\Z$.  In terms of $f_K$, the number field $K$ is locally monogenic at $p$ if and only if $f_K$ represents a unit over $\Z_p$, but this does not imply that $f_K$ represents $\pm1$ over~$\Z_p$. 
\begin{definition}\label{locobs}{\em We say that $K$ {\it has no local obstruction to being monogenic}
if $f_K$ represents~$1$ over $\Z_p$ for all primes $p$ or represents $\minus1$ over $\Z_p$ for all primes $p$.}
\end{definition}
\noindent
See \S\ref{subsec:locmonex} for an example of a number field that is locally monogenic but 
nevertheless has a local obstruction to being monogenic. Note that the index form of a number field $K\neq \Q$ always represents $+1$ and $\minus1$ over $\R$, so we need not consider obstructions over $\R$.

We now specialize to the case of cubic fields $K$,
so that $f_K$ is then an integral binary cubic form.  The following theorem gives the proportion of cubic fields that are locally monogenic, as well as the proportion of those that have no local obstruction to being monogenic.

\begin{theorem}\label{thm2}\label{localthm}
When isomorphism classes of either totally real or complex cubic fields are ordered by absolute discriminant:
\begin{enumerate}
\item[$(a)$] the proportion of these cubic fields that are locally monogenic is $19/21\approx 90.48\%$;

\item[$(b)$] the proportion of these cubic fields that have no local obstruction to being monogenic~is  $$\frac{19}{21}\cdot \frac{316}{351} \cdot \frac{965}{1026}\cdot \prod_{\scriptstyle{p\equiv 1\pmod*6}\atop\scriptstyle{p\neq7}} \left(1-\frac{2}{3(p^2+p+1)}\right) \approx 75.99\%.$$
\end{enumerate}
\end{theorem}
\noindent
Theorem~\ref{S3main} thus shows, for the first time, that a positive proportion of the
set of cubic fields in Theorem~\ref{thm2}(a) and indeed in Theorem~\ref{thm2}(b) are not monogenic.  The analogous result for cubic {orders} was proved by Akhtari and the second author~\cite{AB}, but these orders were not maximal due to the methods employed there
involving the study of Thue equations. 

While previous work on the monogenicity
of general number rings and fields has focused on solving integral Thue equations and index form equations  through techniques of  Diophantine approximation (cf.~\cite{Akhtari2,AB,Bennett,Berczes,EvertseCubic,EvertseSurvey,EG,EGbook,Gaal,Gyori,Okazaki}),  our method instead involves a careful study of rational points on related genus one curves.  
Namely, we study the index forms $f_K$ by comparing  $2$-descent and $3$-descent in the family of Mordell curves $E_k:y^2=x^3+k$, as~$k$ ranges over certain sets of integers defined by congruence conditions, as studied in \cite{leventcircle} and~\cite{j=0}, respectively. 
As a by-product, we prove the following theorem.

\begin{theorem}\label{Hasse}
When $\GL_2(\Z)$-classes of integral binary cubic forms $f(x,y)$ are ordered by absolute discriminant, a positive proportion of the genus one curves $z^3=f(x,y)$ over $\Q$ fail the Hasse principle.

The same statement remains true even if one restricts just to index forms $f_K(x,y)$ of cubic fields $K$:
when isomorphism classes of cubic fields $K$ are ordered by absolute discriminant, a positive proportion of the corresponding genus one curves $C_K \colon z^3 =  f_K(x,y)$ over $\Q$ fail the Hasse principle. 
\end{theorem}
\noindent

Our proof of Theorem~\ref{Hasse} also implies the following statement regarding the Tate--Shafarevich groups $\Sha(E_k)$ of the elliptic curves $E_k$:

\begin{theorem}\label{shatheorem}
Fix $r \geq 0$.  For a positive proportion of $k$, we have $\dim_{\F_3} \Sha(E_k)[3] \geq r$.
\end{theorem}

\noindent
In fact, our proof implies that the same statement remains true if one ranges over just the integers $k$ satisfying any finite set of congruences, or even if one ranges over those integers $k$ that are cubefree and satisfy any further finite set of congruence conditions.

\subsection*{Methods}
To prove Theorem~\ref{S3main}, we must bound how often the Thue equation $f_K(x,y) = 1$ has integer solutions as $K$ ranges over (isomorphism classes of) cubic fields ordered by absolute discriminant.
This equation defines a genus one curve $C_K:z^3=f_K(x,y)$ in~$\P^2$, and an integer solution gives rise to a rational point $P_K$ on the  Jacobian of ~$C_K$, namely the elliptic curve $E^D \colon y^2  = 4x^3 + D$ where $D = \Disc(K)$.  The class of $P_K$ in $E^D(\Q)/3E^D(\Q)$ depends only on $K$, and different cubic fields give rise to different classes. More generally, if $K$ has no local obstruction to being monogenic, then $C_K$ determines a class in the $3$-Selmer group $\Sel_3(E^D)$, and different cubic fields yield distinct Selmer classes \cite[\S1]{satge}.

The motivation and starting point for our proof is recent work of the second and third author and Elkies \cite{j=0}, where it was shown that there is a partition $\bigcup_{m \in \Z} T_m$ of $\Z$,  with each $T_m$ defined by congruence conditions, such that
 if $S$ is any ``large'' subset of $T_m$, then the average size of $\Sel_3(E^D)$, as $D$ varies over $S$, is at least $1 + 3^m$. 
 Hence if $S\subset T_m$ is a set of cubic field discriminants with $m$ large, then as $D$ ranges over $S$, $\Sel_3(E^D)$ is very large on average. 

This suggests 
that we may prove Theorem~\ref{S3main} by showing that: 

\begin{itemize}
    \item[(a)] there are many cubic fields of discriminant $D$, for many  discriminants $D\in S$, that have no local obstruction to being monogenic; and  \item[(b)] the ranks of many of the curves $E^D$ for $D\in S$ are small.
\end{itemize}
Suitably strong and coordinated versions of (a) and (b) would imply Theorem~\ref{S3main}: indeed, (a) would supply a large number of cubic fields across discriminants that have no local obstruction to being monogenic; on the other hand, these cubic fields could not all be monogenic---and thereby yield rational points on elliptic curves $E^D$---due to the ranks of many of these $E^D$ being too small by (b).

Towards (b), we apply a recent result of the first author \cite[Thm.~3.1.1]{leventcircle}---extending a result of Ruth \cite{Ruth} and using a refinement of the circle method due to Heath-Brown~\cite{Heath-Brown}---which states that the average size of the $2$-Selmer groups $\Sel_2(E^D)$, over any large set $S$ of $D$ defined by congruence conditions, is $3$, irrespective of the congruence conditions.  This gives an upper bound on the average size of $2^{\rk\, E^D}$ for $D$ in any large set $S$ defined by congruence conditions. This yields quite a  strong version of (b).

Regarding (a), 
a first natural approach is to try and apply
the asymptotic count of cubic fields due to Davenport--Heilbronn~\cite{DH}. However, this 
does not yield a strong enough version of (a) to show, in conjunction with our version of~(b), that many cubic fields (with no local obstruction to being monogenic) are not monogenic. Indeed, to obtain an upper bound on the globally soluble part of $\Sel_3(E^D)$, we would need an upper bound on the average size of~$3^{\rk\,E^D}$. But $3>2$, so our upper bound on the average size of $2^{\rk\, E^D}$ does not suffice.

To remedy this, we construct explicit large subsets $S \subset T_m$, defined by congruence conditions, with the following two properties. First, each cubic field of discriminant $D \in S$ has no local obstruction to being monogenic. Second, there is a subset $S' \subset S$ of relative density $\mu \geq 1/2$ such that there are {\it at least} $2^{m-1}$ cubic fields of each discriminant $D \in S'$.

The construction of $S$ involves two steps.  First, we use a theorem of the second author and Varma \cite{BV} guaranteeing vanishing of the $3$-parts of class groups of the fields $\Q(\sqrt{\minus3D})$ (when $D>0$) or $\Q(\sqrt{D})$ (when $D<0$) whenever $D\in S' \subset S$, where $S'$ has relative density at least $1/2$ in $S$.  We then use a higher composition law (when $D>0$) and class field theory (when $D<0$) to prove that there are {\it exactly} $2^{m-1}$ (resp.\ $3\cdot 2^{m-1}$) cubic fields of discriminant~$D$ when $D \in S'$ and $D>0$ (resp.\ $D<0$). This yields our desired version of~(a).

Now, were $100\%$ of cubic fields of discriminant $D\in T_m$ monogenic, one would find that the ranks of $E^D(\Q)$ for $D\in S'$ would all be at least $m\log_32$, since we produce at least $2 \cdot 2^{m-1}$ distinct elements $\pm P_K\in E^D(\Q)/3E^D(\Q)$ for each such $D$. This would force the average of $\#\Sel_2(E^D)$, for $D \in S$, to be at least $\mu\cdot 2^{m\log_32}\geq 2^{m\log_32 - 1}$, which for $m$ sufficiently large is strictly larger than $3$, contradicting the aforementioned~\cite[Thm.~3.1.1]{leventcircle}. This proves Theorem \ref{S3main}.

In particular, the Hasse principle thus fails for many genus one curves of the form $z^3=f_K(x,y)$ with $\Disc(K)\in S$, 
thereby proving Theorem~\ref{Hasse}. These curves represent elements of the Tate--Shafarevich group~of $E^{\Disc(K)}$, so that once $m$ is sufficiently large we also deduce Theorem~\ref{shatheorem}.  

\smallskip
This paper is organized as follows.  In \S\ref{sec:localmon}, we prove Theorem~\ref{localthm} by suitably applying the results of Davenport--Heilbronn~\cite{DH} and their refinements in \cite{BST}.  In \S\ref{2sel} and \S\ref{subsec:isogeny-selmer}, we recall from \cite{leventcircle} and \cite{j=0} the relevant definitions and results regarding $2$- and $3$-Selmer groups of the curves $E_k:y^2=x^3+k$, as $k$ varies over suitable congruence families, and we describe their connections to monogenicity.  In \S\ref{tot ram cubics}, we prove our formulas for the number of cubic fields of given  discriminant in $S'$, which we then use in \S\ref{sec:thm1pf} and \S\ref{sec:final} to compare $2$- and $3$-descent on~$E_k$, thus deducing Theorems~\ref{S3main},   \ref{Hasse},  and~\ref{shatheorem}. 

\subsection*{Acknowledgments}
The authors thank Arul Shankar for many helpful comments on earlier versions of this manuscript, and Michael Stoll for organizing Rational Points~2017 and 2019, which provided the pleasant environment where the authors first had the idea for this work. The first author also thanks David Harari, Emmanuel Peyre, and Alexei Skorobogatov for organizing the similarly invigorating Reinventing Rational Points program. The second author thanks Professor John Tate for posing this problem to him when he was in graduate school. The first author was supported by NSF grant~DMS-2002109. The second author was supported by a Simons Investigator Grant and NSF
grant~DMS-1001828. The third author was supported by the Israel Science Foundation (grant No.\ 2301/20).

\addtocontents{toc}{\protect\setcounter{tocdepth}{2}}

\section{The proportion of cubic fields that are locally monogenic, and the proportion that have no local obstruction to being monogenic}\label{sec:localmon}

Let $V(\Z) = \Sym^3\,\Z^2$ be the space of binary cubic forms over $\Z$.  Recall that there is a discriminant-preserving bijection between $\GL_2(\Z)$-orbits on $V(\Z)$ and isomorphism classes of cubic rings over $\Z$, i.e., rings that are free of rank 3 as $\Z$-modules, due to Levi \cite{Levi}, Delone--Faddeev \cite{DF}, and Gan--Gross--Savin~\cite[\S4]{GGS}.  Namely, a cubic ring $R$ corresponds to the $\GL_2(\Z)$-orbit of its index form $f_R \colon R/\Z \to \wedge^3R$. 

The number of $\GL_2(\Z)$-classes of irreducible binary cubic forms having bounded absolute discriminant satisfying certain allowable sets of congruences was determined by Davenport and Heilbronn~\cite{DH}: 

\begin{theorem}[Davenport--Heilbronn~\cite{DH}]\label{dav} 
Let $S$ be a $\GL_2(\Z)$-invariant set of integral binary cubic forms $f$ defined by congruence conditions modulo bounded powers of primes $p$ where, for sufficiently large $p$, the defining congruence conditions at $p$ exclude only a set of forms $f$ satisfying $p^2\mid \Disc(f)$. 
Let $h(S;D)$ denote the number of classes of irreducible binary cubic forms that are contained in $S$ and have discriminant $D$.  Then:

\vspace{.05in}
\begin{itemize}[label=(\alph*)]
\item[$(a)$] $\displaystyle{\sum_{-X <D < 0} h(S;D) \sim \frac{\pi^2}{24}\cdot X\cdot \prod_p\mu_p(S);}$
\item[$(b)$] $\displaystyle{\;\:\sum_{0 <D < X} \,h(S;D) \sim \frac{\pi^2}{72}\cdot X\cdot \prod_p\mu_p(S),}$
\end{itemize}
where $\mu_p(S)$ denotes the $p$-adic density of $S$ in the space of integral binary cubic forms. 
\end{theorem}

To prove Theorem~\ref{thm2}(a) (resp.\ Theorem~\ref{thm2}(b)), we shall apply Theorem~\ref{dav} with $S$ the set of integral binary cubic forms that are index forms of a cubic field and, furthermore, represent a unit over $\Z_p$ for all $p$ (resp.\ represent~$1$ over $\Z_p$ for all $p$). Note that a binary cubic form over a ring $R$ represents $1$ over $R$ if and only if it represents $\minus1$.

\subsection{The density of index forms of cubic fields that locally represent a unit}\label{subsec:loc}
We follow the arguments of \cite[\S4]{BST}, keeping track of those integral binary cubic forms that represent a unit.  We use the following elementary lemma whose proof is immediate:

\begin{lemma}\label{localrep}
 A binary cubic form $f$ over $\Z_p$ represents a unit in $\Z_p$ if and only if $f$ is primitive and furthermore, if $p=2$, then $f(x,y)\not\equiv xy(x+y)\pmod2$.
\end{lemma}
The condition of primitivity is automatically satisfied for index forms of cubic fields. Thus the only modification to the computations of \cite[\S4]{BST} that must be made to determine the density of index forms of cubic fields that locally represent a unit is the elimination of the splitting type $(111)$ at the prime $p=2$. By \cite[Lem.~19]{BST}, the total local $p$-adic density $\mu_p(\mathcal U)$ of the set of index forms is $(p^2-1)(p^3-1)/p^5$ (which is $21/32$ for $p=2$), while by the first line of \cite[Lem.~18]{BST}, the portion of this density corresponding to the splitting type $(111)$ is $\frac16 (p-1)^2(p+1)/p^3$ (which is $2/32$ for $p=2$). Thus the Euler factor for $p=2$ in Theorem~\ref{dav} changes from $21/32$ to $19/32$. We conclude that the proportion of index forms of cubic fields that locally represent a unit is $19/21$.
We have proven Theorem~\ref{thm2}(a).

\subsection{The density of index forms of cubic fields that locally represent 1} We use the following lemma, which is a slight correction of \cite[Lem.~4.1]{AB} at the prime $p=7$:

\begin{lemma}\label{localrep1}
 A primitive binary cubic form $f$ over $\Z_2$ represents $1$ if and only if $f(x,y)\not\equiv xy(x+y)\pmod2$.
 If $p\equiv 1\pmod6$, then a primitive binary cubic form $f$ over $\Z_p$ represents~$1$ if and only if 
 $f(x,y)\not\equiv c L(x,y)^3\pmod p$, where $L$ is a linear form and $c$ is a non-cube modulo $p$, and furthermore, if $p=7$, then $f(x,y)$ is not equivalent to $2 xy(x+y)\pmod7$ under a linear change of variable.  Finally, 
 if $p\equiv 5\pmod 6$, then every primitive binary cubic form $f$ over $\Z_p$ represents~$1$. 
\end{lemma}
\begin{proof}
 The proof of \cite[Lem.~4.1]{AB} is correct except for a small oversight in the case $p\equiv1\pmod6$ when $p=7$. In that case, the Hasse-Weil bound shows that a smooth genus one curve $z^3=f(x,y)$ has at least 3 points over $\F_7$, but these three points 
might all satisfy $z=0$, so that $f(x,y)$ might {\it not} represent a unit cube, and therefore $1$, over $\F_7$.  Such a scenario can occur only if $f(x,y)$ splits completely over $\F_7$. All such forms over $\F_7$ are $\GL_2(\F_7)$-equivalent to $cxy(x+y)$, for some $c\in\F_7^\times$. Since $c$ can be scaled by any unit cube by scaling $x$ and $y$ by a unit, we may assume that $c = 1$, $2$, or $3$, as $\pm 1$ are the only cubes in $\F_7^\times$. Of these three possibilities, only $c=2$ yields a binary cubic form over $\F_7$ that does not represent a unit cube.
\end{proof}

We also have the following analogue of Lemma~\ref{localrep1} for index forms over $\Z_3$:
\begin{lemma}\label{lem:p=3}
The index form $f$ of a maximal cubic ring over $\Z_3$ represents $1$ over $\Z_3$ if and only if $f$ is not equivalent over $\Z/9\Z$ 
to 
$cxy(x+y)$,  $2cx^3-3xy^2+3y^3$, or $2cx^3-3x^2y+3y^3$ for $c=1$ or $2$.
\end{lemma}

\begin{proof}
The cubes in  $\Z_3^{\times}$ are $\pm1 + 9\Z_3$, so it suffices to consider the index forms modulo~$9$.  
Forms of type $(111)$ are equivalent to $cxy(x+y)$, which represents $1$ only when $c\equiv \pm4\pmod*{9}$. Meanwhile, index forms of type $(12),$ $(3)$, and $(1^21)$ are equivalent to $x(x^2 + y^2)$, $x^3 - xy^2 +y^3$, and $x^2y$, respectively, and these all represent~$1 \pmod*{9}$.  

Index forms of type $(1^3)$ can naturally be  subdivided into $3$ subtypes based on the $3$-adic valuations of their discriminants.
Those with   valuation $3$ are equivalent to $c(x^3 + 3xy^2 + 3y^3)$ or $c(x^3  -3xy^2 + 3y^3)$, and these forms represent~$1$ unless $c$ is $\pm2$ or $\pm4$ (mod~$9$), respectively. Index forms $f$ of valuation $4$ come in two types.  
If $\Gal(f) \simeq S_3$, then, up to equivalence, $f(x,y) \equiv c(x^3 + 3x^2y + 3y^3) \pmod*{9}$, which represents $1$ (mod~$9$).  If $\Gal(f) \simeq C_3$, then up to equivalence $f \equiv c(x^3 - 3x^2y + 3y^3) \pmod 9$, which represents $1$ if and only if $c \equiv \pm1 \pmod 9$. Finally, those with valuation $3^5$ are equivalent to $ax^3 + 3dy^3$, with $a$ and $d$ units, which all represent~$1$ (mod~$9$).  
\end{proof}

To obtain the relative density at $p$ of index forms that locally represent~$1$, we use Lemmas~\ref{localrep1} and \ref{lem:p=3}. At $p=2$, we have already seen that the relative density is $19/21$. 

For primes $p\equiv 1 \pmod*{6}$ but $p\neq7$, we use the fact that a relative density of $2/3$ of index forms that have splitting type $(1^3)$ (mod $p$) do not represent~$1$ over $\Z_p$. By \cite[Lem.~18]{BST}, a density of $(p -1)^2 (p + 1) / p^5$ of binary cubic forms are index forms locally and have splitting type $(1^3)$. Thus the desired relative density of index forms that represent~$1$ over $\Z_p$, for primes $p\pmod*{6}$ and $p\neq 7$, is 
\[1-\frac23\cdot\frac{(p -1)^2 (p + 1) / p^5}{(p^2-1)(p^3-1)/p^5} =1-\frac{2}{3(p^2+p+1)} .\]

If $p=7$, we must further subtract the relative density of index forms that split completely over $\F_7$ and do not represent~$1$ over $\Z_7$, which as we have seen is $1/3$ of all forms that split completely over $\F_7$. The desired relative density of index forms that represent~$1$ over $\Z_7$ is thus
\[1-\frac{2}{3(p^2+p+1)}
- \frac13\cdot \frac{(1/6)(p -1)^2 (p + 1) / p^3}{(p^2-1)(p^3-1)/p^5} =
1-\frac{p^2+12}{18(p^2+p+1)}=\frac{965}{1026}\]
when $p=7$.

If $p\equiv 5 \pmod*{6}$, then every index form represents $1$ over $\Z_p$. 

Finally, if $p = 3$, then $1/3$ of forms of type $(111)$ represent~$1$ over~$\Z_3$ by Lemma \ref{lem:p=3}.  Among index forms of type $(1^3)$, those with discriminant of valuation $5$ have density~${1}/{9}$ since they have a triple root modulo 9, not just modulo 3, and these all represent 1 over $\Z_3$. 
Among the remaining $8/9$ of forms of type $(1^3)$, we claim that a proportion of $2/3$ represent~$1$ over $\Z_3$.  This is true for those with discriminant of valuation $3$ by Lemma~\ref{lem:p=3}. 
 There are three cubic extensions of $\Q_3$ with discriminant of valuation~$4$ with Galois group~$C_3$ and one with Galois group~$S_3$.  Thus, when weighted appropriately, there is the same proportion of each type.  
By the proof of Lemma~\ref{lem:p=3}, the proportion of these forms that represent~$1$ over $\Z_3$ is $\frac12\left(\frac13 + 1\right) = \frac23$, as claimed. Thus the desired relative density of index forms $f_K$ that represent~$1$ over $\Z_3$ is 
\[1 - \frac23\cdot \frac{(1/6)(p -1)^2 (p + 1) p^2}{(p^2-1)(p^3-1)} - \frac{1}{3}\cdot\frac{8}{9}\cdot\frac{(p-1)^2(p+1)}{(p^2 - 1)(p^3 - 1)}= 
\frac{316}{351}
\]
when $p=3$. We have proven Theorem~\ref{thm2}(b). 
\subsection{A cubic field that is locally monogenic but has a local obstruction to being monogenic}\label{subsec:locmonex}

The cubic field $K = \Q[X]/(X^3 - 7/5)$ has  ring of integers $\O_K = \Z + \Z\cdot 5\theta + \Z\cdot 5\theta^2$, where $\theta$ is the image of $X$.  The index $f_K(x,y)$ of  $x\cdot 5\theta + y\cdot 5\theta^2$ is given by $f_K(x,y) = 5x^3 - 7y^3$, which does not represent~$\pm1$ over $\F_7$, let alone over $\Z_7$ or~$\Z$.  However, $K$ is locally monogenic; for example, $\O_K \otimes_\Z \Z_7 = \Z_7[\theta]$.  Thus $K$ has a local obstruction to being monogenic, despite being locally monogenic.  

On the other hand, consider $K' = \Q[X]/(X^3 - 21)$, which has a monogenic ring of integers $\O_{K'} = \Z[\sqrt[3]{21}]$.  Then $f_{K'}(x,y) = x^3 - 21y^3$, which represents~$1$ over~$\Z_7$ and indeed over~$\Z$. Note that $\O_K \otimes_\Z \Z_7 \simeq \O_{K'} \otimes_\Z \Z_7$, but $K$ has a local obstruction to being monogenic over~$\Z_7$ while $K'$ does not.  Thus the property of having a local obstruction to being monogenic is not solely a property of the  localization $K \otimes_\Q \Q_p$, but is also a property of the localization of the global index form. 

\begin{remark}
{\em Isomorphism classes of cubic rings over a principal ideal domain $R$ (i.e., $R$-algebras that are free of rank $3$ as modules over $R$) are parametrized by orbits of binary cubic forms over $R$ under the ``twisted'' action of $\GL_2(R)$ rather than the ``standard'' action; see~\cite[Prop.~2.1]{GL}. The {\it standard} action of $\gamma\in\GL_2(R)$ on a binary cubic form $f$ with coefficients in $R$ (which we use in this paper throughout) is given~by \begin{equation}\label{gl2action}
\gamma\cdot f(x,y):=f((x,y)\gamma),
\end{equation} while the
{\it twisted} action is given by
\begin{equation}\label{gl2action2}
\gamma\star f(x,y):=\det(\gamma)^{\minus1}f((x,y)\gamma),
\end{equation} 
where we view $(x,y)$ as a row vector. 

For binary cubic forms over $R=\Z$, the standard and twisted $\GL_2(R)$-orbits are identical, but for $R=\Z_p$ they may not be. The property of whether or not a binary cubic form $f$ represents $1$ over $\Z_p$ is preserved under standard $\GL_2(\Z_p)$-equivalence but not necessarily under twisted $\GL_2(\Z_p)$-equivalence. This explains, in terms of binary cubic forms, why two cubic fields may have isomorphic completions at a prime $p$, and both may be locally monogenic at $p$, but only one may have a local obstruction to being monogenic at $p$. 

The cubic fields $K$ and $K'$ above have this property; the corresponding binary cubic forms $5x^3-7y^3$ and $x^3-21y^3$, when viewed in $\Z_7[x,y]$, are not standard $\GL_2(\Z_7)$-equivalent, and indeed represent different cube classes in $\Z_7$, but they are twisted $\GL_2(\Z_7)$-equivalent, and thus correspond to the same cubic ring over $\Z_7$.}
\end{remark}

\section{A positive proportion of cubic fields have no local obstruction to being monogenic but are not monogenic}

\subsection{$2$-Selmer groups of the elliptic curves $E^D$}\label{2sel}
Recall that the $2$-Selmer group of an elliptic curve $E$ over $\Q$ is defined by
\[\Sel_2(E) := \ker\biggl(H^1(\Q, E[2]) \to \prod_{p \leq \infty}^{\phantom{.}} H^1(\Q_p, E)\biggr)\]
and sits in an exact sequence
\[0 \to E(\Q)/2E(\Q) \to \Sel_2(E) \to \Sha(E)[2]\to 0.\]
In particular, we have $\#\Sel_2(E)\geq 2^{\rk\, E}\cdot \#E(\Q)[2]$, where $\rk \, E$ is the rank of the finitely-generated abelian group $E(\Q)$.  

Recall the elliptic curve $E^D \colon y^2 = 4x^3 + D$.  We require a bound on the average size of $\Sel_2(E^D)$ as $D$ varies through certain sets $\Sigma \subset \Z$ defined by congruence conditions.  We define the {\it average} $\avg_\Sigma(f)$ of a function $f$ on $\Sigma$ with respect to the natural density, i.e.,
\[\avg_\Sigma(f) = \lim_{X \to \infty} \frac{1}{\#\Sigma(X)}\sum_{D \in \Sigma(X)} f(D),\]
where $\Sigma(X) := \Sigma\cap [-X, X]\setminus\{0\}$.  

\begin{definition}{\em 
A set $\Sigma \subset \Z$ is {\it defined by congruence conditions} if there is a set $\Sigma_\infty\in\{\R_{>0},\R_{<0},\R^\times\}$ and for each prime $p$ there is an open and closed subset $\Sigma_p \subset \Z_p$ such that $\Sigma =  \cap_p(\Z \cap \Sigma_p)\cap\Sigma_\infty$.  If, moreover, $\Sigma \neq \emptyset$ and $$\Sigma_p \supset \{D \colon v_p(D) \leq 1\}$$ for all sufficiently large $p$, then $\Sigma$ is called {\it large}.}  
\end{definition}

The following is a special case of a theorem of the first author \cite[Thm.~3.1.1]{leventcircle}, which in turn is an extension of Ruth's Ph.D. thesis \cite{Ruth} who proved that $\mathrm{avg}_{\Z}\, \#\Sel_2(E^D)\leq 3$.

\begin{theorem}\label{alpoge2Sel}
Let $\Sigma \subset \Z$ be a large set defined by congruence conditions. Then $$\mathrm{avg}_{\Sigma}\, \#\Sel_2(E^D) = 3.$$
\end{theorem}

\begin{corollary}\label{cor:averagerankbound}
The average rank of $E^D$, as $D$ ranges over any large set $\Sigma$, is at most $3/2$.
\end{corollary}
\begin{proof}
We have $\avg_\Sigma (2\,\rk\, E^D)\leq\avg_\Sigma(2^{\rk\, E^D}) \leq \avg_\Sigma \#\Sel_2(E^D)=3$.
\end{proof}

\subsection{$3$-isogeny Selmer groups of the elliptic curves $E^D$}\label{subsec:isogeny-selmer}
Let $F$ be a field of characteristic not $2$ or $3$.  
For any nonzero $D \in F$, there is a degree-$3$ isogeny $\phi_D \colon E^D \to E^{\minus27D}$ 
whose kernel is generated by~$(0, \pm \sqrt{D})$.  More generally, if $f \in V(F)$ is a binary cubic form of discriminant~$D$, then the plane cubic curve $C_f \colon z^3  = f(x,y)$ admits the degree 3 map $\phi_f \colon C_f \to E^{\minus27D}$ given by
\[[x \colon y \colon z] \mapsto \left(\frac{\minus h(x,y)}{z^2},\frac{g(x,y)}{z^3}\right),\]
where $h(x,y) = \frac{1}{4}\left(f_{xx}f_{yy} - f_{xy}^2\right)$
is the Hessian of $f$ and $g(x,y) = f_xh_y - f_yh_x$ is the Jacobian derivative of $f$ and $h$ (see  \cite[Rem.~31]{j=0}). 
When $f(x,y) = x^2y -(D/4)y^3$, or more generally when $f$ is in the unique $\SL_2(F)$-orbit of {\it reducible} forms of discriminant $D$, we have $C_f \simeq E^D$ and we define $\phi_D = \phi_f$.  Since any $f$ is reducible over the separable closure $\bar{F}$, we see that $E^D$ is the Jacobian of $C_f$, and $\phi_f$ is a twist of $\phi_D$.  If $K/F$ is a cubic extension and $f = f_K$ is the corresponding index form, we write $\phi_K = \phi_f$ and $C_K = C_f$. 

Since $\Aut(\phi_D) = E^D[\phi_D]$, the twists of $\phi_D$ are parametrized by the $\F_3$-vector space $H^1(F, E^D[\phi_D])$.  Thus, each cubic extension $K/F$ with discriminant (in the square class of) $D$ gives rise to a pair of nonzero classes $\pm \alpha_K \in H^1(F, E^D[\phi_D])$  corresponding to the class of $\phi_K$.  There are two classes because there are two $\SL_2(F)$-orbits of discriminant $D$ in the $\GL_2(F)$-orbit of $f$. The genus one curves corresponding to $\alpha_K$ and $\minus\alpha_K$ are isomorphic, so we can unambiguously write $C_K$.

\begin{lemma}\label{distinct}
If $K$ and $K'$ are non-isomorphic cubic extensions of $F$ of the same discriminant~$D$, then $\alpha_K$ and $\alpha_{K'}$ are linearly independent in $H^1(F, E^D[\phi_D])$.
\end{lemma}
\begin{proof}
Recall that $K \simeq F[x]/(f_K(x,1))$.  Thus $f_K$ is reducible over $K$, while $f_{K'}$ stays irreducible over $K$. Since the reducible orbit corresponds to the trivial class, this means $\alpha_K$ is in the kernel of $H^1(F, E^D[\phi_D]) \to H^1(K, E^D[\phi_D])$, while $\alpha_{K'}$ is not.
\end{proof}

\noindent
Finally, we remark that the natural map $H^1(F, E^D[\phi_D]) \to H^1(F, E^D)$ sends the class of $\phi_f$ to the class of the twist $C_f$ of $E^D$.  

Now take $F = \Q$. The $\phi_D$-Selmer group is defined by
\[\Sel_{\phi_D}(E^D) := \ker\biggl(H^1(\Q, E^D[\phi_D]) \to \prod_p^{\phantom{.}} H^1(\Q_p, E^D)\biggr).\]
The classes $\pm\alpha_K$ lie in $\Sel_{\phi_D}(E^D)$ if and only if $C_K(\Q_p) \neq \emptyset$ for all primes $p$.

The inclusion $E^D[\phi_D] \subset E^D[3]$ induces a map $\Sel_{\phi_D}(E^D) \to \Sel_3(E^D)$, which is injective if $\ker(\hat\phi_D)(\Q) = 0$, i.e., if $\minus27D$ is not a perfect square \cite[Eq.~(22)]{j=0}.  
From the exact sequence
\[0 \to E^D(\Q)/3E^D(\Q) \to \Sel_3(E^D) \to \Sha(E^D)[3] \to 0,\]
we see that any class in $\Sel_{\phi_D}(E^D)$ such that $C_f(\Q) \neq \emptyset$ lies in the subgroup  $E^D(\Q)/3E^D(\Q)$.  In particular, if $\minus27D$ is not a perfect square and $K$ is a monogenic cubic field of discriminant~$D$,  then $E^D$ has rank at least $1$.

\begin{remark}{\em 
If $D = \minus27n^2$, then the kernel of $\Sel_{\phi^D}(E^D) \to \Sel_3(E^D)$ has size $3$, generated by the class of $\phi_f$ with $f(x,y) = x^3 +ny^3$. Its Hessian $h$ is $9nxy$, and hence   $\phi_f(1 \colon 0 \colon 1)$ is the rational 3-torsion point $(0,27n)$ on $E^{\minus27D} =  E^{3^6n^2}$. In this special case, $E^D$ can have rank~$0$ despite there being a monogenic cubic field of discriminant $D$; this occurs, for example, when $n = 5$.}   
\end{remark}

For any large subset $\Sigma \subset \Z$, the second and third authors and Elkies computed the average size of $\Sel_{\phi_D}(E^D)$ for $D \in \Sigma$ in \cite{j=0}. Unlike Theorem \ref{alpoge2Sel} for the 2-Selmer group, this average size is sensitive to the congruence conditions defining $\Sigma$. However, one can partition $\Z$ into sets $T_m$ ($m \in \Z$) such that each $T_m$ has positive density and, for any large subset $\Sigma \subset T_m$, the average size $\avg_\Sigma \#\Sel_{\phi_D}(E^D)$ is independent of $\Sigma$.

\begin{theorem}[\cite{j=0}]\label{thm:phiselavg}
Fix $m \in \Z$ and let $\Sigma \subset \Z$ be a large set contained in $T_m$.  Then  $$\avg_\Sigma\#\Sel_{\phi_D}(E^D)=1 + 3^m.$$
\end{theorem}

The sets $T_m$ are defined by infinitely many congruence conditions; we refer to \cite{j=0} for the exact definition. The key point is that in order for $D$ to lie in $T_m$, it must be divisible by $n^2$ for some $n$ divisible by at least $|m|-1$ primes $p \equiv 2 \pmod* 3$.  

This suggests that, to prove the existence of many non-monogenic cubic fields, one might focus on cubic fields of discriminant $D = dn^2$, with $n$ divisible by many primes $p \equiv 2\pmod* 3$.    

\subsection{Construction of cubic fields ramified at a fixed set of primes with no local obstruction to being monogenic}
\label{tot ram cubics}
To prove Theorem \ref{S3main}, we construct a subset $\Sigma \subset \Z$ for which we can guarantee the following: 

\begin{enumerate}
    \item[(i)] a cubic field $K$ with $\Disc(K) \in \Sigma$ has no local obstruction to being monogenic; and 
    \item[(ii)] for a positive proportion of $D \in \Sigma$, there are ``many'' cubic fields of discriminant $D$.   
\end{enumerate}

\vspace{.02in}
\noindent
To construct such $\Sigma$, we choose $t$ distinct primes $p_1, \ldots, p_t$ each congruent to $2 \pmod* 3$, and set $n  = 3\cdot \prod_{i = 1}^t p_i$. 
It will be convenient to assume $p_1 = 2$, and we do so. For Theorem \ref{S3main} we will simply take $t = 2$ and $n = 2 \cdot 3 \cdot 5 = 30$, but for Theorem \ref{shatheorem} we will take $t$ to be arbitrarily large.

Define the set of integers
\[
\Sigma_n = \left\{\minus27dn^2 \in \Z \colon  d \mbox{ fundamental}, \, \left(\frac{d}{7}\right) \neq 1,  \mbox{ and } \left(\frac{d}{p}\right) = 1 \, \mbox{ for all } \, p \mid n \right\},
\]
where $\Bigl(\frac{d}{p}\Bigr)$ denotes the Kronecker symbol. Note that $\Sigma_n$ has positive natural density.

\begin{lemma}\label{lem:unob}
If $K$ is a cubic field and $\Disc(K) \in \Sigma_n$, then $f_K$ represents~$1$ over~$\Z_p$ for all~$p$, i.e., $K$ has no local obstruction to being monogenic.
\end{lemma}

\begin{proof}
This follows from Lemmas \ref{localrep1} and \ref{lem:p=3}. For $p\equiv5 \pmod*6$ there is nothing to check. If $p\equiv1\pmod*3$ then $p^2 \nmid \Disc(f_K)$, and so $f_K$ represents $1$ over $\Z_p$ for all such $p > 7$.  The $3$-adic and $7$-adic conditions imposed on $d$ guarantee that $f$ represents 1 over $\Z_3$ and $\Z_7$.  Finally, $f_K$ represents $1$ over $\Z_2$ since it is primitive and $2 \mid \Disc(f_K)$. 
\end{proof}

Let $\Sigma_n^\pm$ be the subset of $D \in \Sigma_n$ such that $\pm D > 0$. The next proposition shows that there are many cubic fields per discriminant in $\Sigma_n^\pm$, on average.  

\begin{proposition}\label{prop:averagecount}
Let $\Sigma_n^\pm(X) = \Sigma_n^\pm \cap [\minus X,X]$, and let $N(\Sigma_n^\pm,X)$ be the number of cubic fields with discriminant in $\Sigma_n^\pm(X)$. Then
\begin{align*}
\lim_{X \to \infty} \dfrac{N(\Sigma_n^\pm, X)}{\#\Sigma_n^\pm(X)} &=(2 \mp1) 2^t, 
\end{align*}
\end{proposition}

\noindent Proposition~\ref{prop:averagecount}  may be deduced from \cite[Thm.~8]{BST}; we omit the details, as it is only suggestive of how to proceed but will not be used  directly in what follows.

As explained in the introduction, average results such as Proposition \ref{prop:averagecount} are not strong enough for our purposes.  In the next two propositions, we provide exact formulas for the number of cubic fields of discriminant $D \in \Sigma_n$ for specific values of $D$, namely those $D$ for which the 3-torsion in the class group of $\Q(\sqrt{D^*})$ is trivial, where $D^* = D$ or $\minus3D$ depending on whether $D< 0$ or $D > 0$.
Moreover, we prove that the average values in Proposition~\ref{prop:averagecount} are actually attained on the nose for these individual discriminants.

\begin{proposition}\label{prop:compositionlaws}
Let $D \in \Sigma_n^+$ and write $F = \Q(\sqrt{\minus3D})$. If $\Cl(F)[3] = 0$, then there are exactly $2^t$ cubic fields of discriminant $D$.
\end{proposition}

\begin{proof}
Let $K$ be a cubic field of discriminant $D = \minus27dn^2 \in \Sigma_n^+$. 
Since $3^5 \mid \Disc(f_K)$, $f_K$ has a triple root over $\F_3$, and hence its central coefficients are divisible by $3$.  Let $S$ be the quadratic ring $\O_{dn^2}$ of discriminant $dn^2$, and let $(S,I,\delta)$ be the triple corresponding to $f_K$ under the bijection of \cite[Thm.~13]{hcl1}.  Thus $I$ is an $S$-ideal such that $I^3 \subset \delta S$ and $\Nm_S(I)^3 = \Nm(\delta)$. 

The condition that $f_K$ corresponds to a maximal cubic order is equivalent to the condition that $I$ is also an ideal for the maximal order $\O_d = \O_{F}$.  To see this, note that the latter is equivalent to the condition that the primitive part of the Hessian $h$ of $f_K$ has discriminant~$d$; see \cite[Eq.~(24)]{hcl1}.  Then note that $\Disc(h) = dn^2$ and $h$ is a scalar multiple of $n$ if and only if $f_K$ has a triple root modulo each $p \mid n$.  

The conditions $I^3 \subset \delta S$ and $\Nm_S(I)^3 = \Nm(\delta)$ are equivalent to the single condition  $I^3= nJ(\delta)$, for some $\O_{d}$-ideal $J$ of norm $n$. 
Indeed, $\Nm_S(I)^3 = \Nm(I)^3/n^3$, so $\a := I^3\delta^{\minus1}$ is an $\O_{d}$-ideal of norm $n^3$ contained in $S$.  By the defining conditions of $\Sigma_n^+$, all primes $p \mid n$ split as $\p\bar\p$ in $\O_{d}$.  Since $\p \otimes_\Z \Z_p \neq S \otimes_\Z \Z_\p \neq \bar \p \otimes_\Z \Z_p$, the ideal $\a$ must be divisible by $n$. Thus $I^3(n\delta)^{\minus1}$ is an $\O_{d}$-ideal of norm $n$.  
Thus each cubic field $K$ of discriminant $D$ gives rise to an $\O_{d}$-ideal $J$ of norm $n$, together with a class $[I] \in \Cl(\O_d)$ which cubes to $[J]$.

Conversely, given a pair $(J, [I])$, where $J$ is an $\O_d$-ideal of norm $n$ and $[I]$ is an ideal class such that $[I]^3 = [J]$, we may write $I^3 = nJ(\delta)$, for some $\delta$. Since $S^\times = \{\pm 1\}$, this $\delta$ is uniquely determined up to sign. Then the equivalence class of $(S,I,\delta)$ corresponds to a $\GL_2(\Z)$-orbit of cubic forms, which itself corresponds to a maximal cubic order of discriminant $D = \minus27dn^2$.  Thus, the bijection of \cite[Thm.~13]{hcl1} induces a bijection:
\begin{align*}
    \{\mbox{cubic fields }K \colon &\Disc(K) = \minus27dn^2\}\\
    &\textstyle\left\updownarrow\vphantom{\int_a^c}\right.\\
\{(J,[I]) \colon J \mbox{ is an $\O_d$-ideal}
 &\mbox{ of norm } n \mbox{ and }[I]^3 = [J]\}/\sim,
\end{align*}
where $\sim$ denotes $\Gal(\O_d/\Z)$-conjugation.   When $\Cl(F)[3] = 0$, every ideal class has a unique cube root, hence the data of $[I]$ is unnecessary.  Since each $p \mid n$ splits in $\O_d$, there are $\frac12\cdot 2^{\omega(n)} = 2^{t}$ equivalence classes of ideals $J$ of norm $n$, and hence $2^t$ cubic fields of discriminant $D = \minus27dn^2$.
\end{proof}

\begin{proposition}\label{prop:cft}
Let $D \in \Sigma_n^-$ and write $F = \Q(\sqrt{D})$. If $\Cl(F)[3] = 0$, then there are exactly $3 \cdot 2^{t}$ cubic fields of discriminant $D$.
\end{proposition}
\begin{proof}
Write $D = \minus27dn^2$.  By class field theory, the number of cubic fields with discriminant of the form $\minus3dm^2$, for some $m \mid 3n$, is the number of index 3 subgroups in the ring class group $\Cl(\O_D)$ of discriminant $D$ \cite[\S5.1]{BV}.  Since $\Cl(F)[3] = \Cl(\O_{\minus3D})[3] = 0$ and $D < 0$, the group $\Cl(\O_D)[3]$ is isomorphic to the 3-torsion subgroup of
\[ \dfrac{(\O_F/3n)^{\times}}{(\Z/3n)^\times} \simeq \dfrac{(\O_F/9)^\times}{(\Z/9)^\times} \times \prod_{i = 1}^t \dfrac{(\O_F/p_i)^\times}{(\Z/p_i)^\times}.\]
Note that each $p_i$ is inert in $F$, by definition of $\Sigma_n$, whereas $3$ ramifies. The $3$-torsion subgroup of the right hand side is therefore isomorphic to $(\Z/3\Z)^2 \times (\Z/3\Z)^t$, since $p_i \equiv 2 \pmod* 3$. Thus, $\Cl(\O_D)$ has $\frac12(3^{t+2} - 1)$ index~3 subgroups, and there are that many cubic fields with discriminant of the form $\minus 3dm^2$ for some~$m \mid 3n$.  By M\"{o}bius inversion, the number of cubic fields of discriminant $D = \minus 27dn^2$ is $3 \cdot 2^t$, as claimed. 
\end{proof}

\begin{remark}{\em
One can also prove Proposition~\ref{prop:cft} using \cite{hcl1}, but the proof is more involved since one must grapple with units in $\Q(\sqrt{\minus3D})$.  Similarly, one can also prove Proposition~\ref{prop:compositionlaws} using class field theory, but again the proof is more involved since one must now deal with units in $\Q(\sqrt{D})$; see \cite[Thm.~2.5]{CT}.  In fact, the method we use to prove Proposition \ref{prop:compositionlaws}  can be used to give an alternative expression for the generating series $\Phi_d$ of \cite{CT}, namely, as a finite sum of what are essentially Epstein zeta functions.
}\end{remark}

\subsection{Proof of Theorem \ref{S3main}}\label{sec:thm1pf}
We prove the stronger statement that a positive proportion of totally real (resp.\ complex) cubic fields are not monogenic, despite having no local obstruction to being so.  
Fix $n$ as in the previous section and note that both $\Sigma_n^+$ and $\Sigma_n^-$ have positive natural density and are large in the terminology of Section~\ref{2sel}. 

If $D \in \Sigma_n^\pm$ and $K$ is a cubic field with $\Disc(K) = D$, then the class $\alpha_K$ defined in Section \ref{subsec:isogeny-selmer} lies in $\Sel_{\phi_D}(E^D)$ by Lemma \ref{lem:unob}.  Since $\minus27D$ is not a perfect square, we have $E^D[\hat\phi_D](\Q) = 0$ and hence an injection $\Sel_{\phi_D}(E^D) \hookrightarrow \Sel_3(E^D)$. We may therefore view $\alpha_K$ as an element of $\Sel_3(E^D) \subset H^1(\Q, E^D[3])$. The image of $\alpha_K$ in $\Sha(E^D)[3]$ vanishes if and only if $C_K(\Q) \neq \emptyset$. In particular, $C_K(\Q)\neq \emptyset$ if $K$ is monogenic, since $f_K(x,y) = 1$ has a solution over $\Z$ and thus over $\Q$. Therefore, when $K$ is a monogenic cubic field of discriminant $D$, the classes $\pm \alpha_K$ lie in the subgroup $E^D(\Q)/3E^D(\Q)$ of $\Sel_3(E^D)$.

Let $U_n^+$ (resp.\ $U_n^-$) be the subset of integers $D \in \Sigma_n^+$ (resp.\ $\Sigma_n^-$) such that the $3$-torsion in the class group of $\Q(\sqrt{\minus3D})$ (resp.\ $\Q(\sqrt{D})$) is trivial.  The lower natural density $\mu_{n,\pm}$ of $U_n^\pm$ within $\Sigma_n^\pm$ is at least $1/2$ by \cite[Thm.~4]{BV}. By Lemma \ref{distinct} and Propositions \ref{prop:compositionlaws} and  \ref{prop:cft}, there are $(2\mp1) 2^{t+1}$ distinct non-zero classes $\pm \alpha_K\in \Sel_3(E^D)$ for each $D \in U_n^\pm$. Write $m_D$ for the number of monogenic cubic fields of discriminant $D$. Thus $0\leq m_D\leq (2 \mp1) 2^{t+1}$ for $D\in U_n^\pm$, and
$1 + 2m_D\leq 3^{\rank{E^D(\Q)}}$.

Switching to $2$-descent, this implies that $$(1 + 2 m_D)^{\log_3{2}}\leq 2^{\rank{E^D(\Q)}}\leq \#\Sel_2(E^D).$$ By the above inequality and  Theorem \ref{alpoge2Sel},
\begin{equation*}
 \mu_{n,\pm}\cdot \avg_{D\in U_n^\pm}(1 + 2 m_D)^{\log_3{2}}
+ (1 - \mu_{n,\pm})\cdot 1 \:\leq\: \avg_{D\in \Sigma_n^\pm} \#\Sel_2(E^D)
 \:=\:3.
\end{equation*}
Therefore, $$\avg_{D\in U_n^\pm} (1 + 2 m_D)^{\log_3{2}}\leq 1 + 2 \mu_{n,\pm}^{\minus1}\leq 5.$$
\noindent

The function $f(x):=x^{\log_23}$ is strictly convex on $\R_{>0}.$ 
Thus, given the two constraints (a) $m_D\in [0,(2\mp 1)2^t]$ for all $D\in U_n^\pm$, and (b) $\avg_{D\in U_n^\pm}(1+2m_D)^{\log_32}\leq 5$, the average $\avg_{D\in U_n^\pm}(1+2m_D)$ is 
maximized when $m_D$
takes all its values at the endpoints of  $[0,(2\mp1)2^t]$. Once $t\geq 3$, this occurs when $m_D$
attains its maximum value $(2\mp1)2^t$ for a density of $4/((1+(2\mp1)2^{t+1})^{\log_32}-1 )=O(2^{\minus t\log_32})$ of $D\in U_n^\pm$,
and $m_D=0$  otherwise.
It follows that a (positive)  lower density of at least $1-O(2^{\minus t\log_32})$ of cubic fields~$K$ with $\Disc(K)$ in $U_n^\pm$ are non-monogenic. 

Since the relative lower density $\mu_{n,\pm}$ of $U_n^\pm$ in $\Sigma_n^\pm$ is at least $1/2$, we have produced  $c\cdot X + o(X)$ (for an explicit constant $c>0$) cubic fields $K$ with $\Disc(K)\in \Sigma_n^\pm$ and $|\Disc(K)|< X$ that are not monogenic and have no local obstruction to being monogenic.
By the asymptotic count of all cubic fields due to Davenport and Heilbronn~\cite{DH}, this proves Theorem \ref{S3main} for both positive and negative discriminants.  
\qed

\begin{remark}\label{explicitconstants}
{\em
For negative discriminants, we can even take $t = 1$ in the argument above.  In that case, we have $\dim_{\F_2} \Sel_2(E^D) \equiv \dim_{\F_3}\Sel_3(E^D) \equiv 1\pmod*2$ for all $D\in \Sigma_n^-$  by the computation of the parity of the 3-Selmer rank in  \cite[Prop.~34, 49(b)]{j=0} and the $p$-parity theorem \cite{DD2}. Thus $\# \Sel_2(E^D)\geq 8$ when $m_D\geq 2$ and $\# \Sel_2(E^D)\geq 2$ always. Reasoning as in the proof of Theorem \ref{S3main}, it then follows that at least $\frac{5}{9}$ of cubic fields with discriminant in $U_n^-$ are non-monogenic. If $n = 2\cdot 3$, we find that there are at least
 \[6 \cdot \frac59 \cdot \mu_{n,-} \cdot \mu(\Sigma_n^-)\cdot X \geq 6\cdot \frac{5}{9}\cdot \frac12\cdot \frac{1}{2^4}\cdot\frac{1}{3^6}\cdot\frac{27}{7^2}\cdot \prod_{p\nmid42}(1 - p^{\minus 2})\cdot X\]
such fields, where $\mu(\Sigma_n^-)$ denotes the natural density of $\Sigma_n^-$.  Dividing by the total number  $\frac{1}{4}\cdot\zeta(3)^{\minus1}\cdot(.7599\ldots)\cdot X$ of unobstructed complex cubic fields (Theorem \ref{localthm}), we conclude that the proportion of unobstructed complex cubic fields that are non-monogenic is at least~$.000463$.

For positive discriminants, we may similarly take $t = 2$ and $n = 2 \cdot 3 \cdot 5$ by parity considerations, and find that a proportion of at least $.0000139$ of unobstructed totally real cubic fields are not monogenic.  Of course, we have not tried to optimize these lower bounds.
}
\end{remark}

\subsection{Proofs of Theorems~\ref{Hasse} and  \ref{shatheorem}}\label{sec:final}

The proof of Theorem~\ref{S3main} immediately implies Theorem~\ref{Hasse}: we have produced a positive proportion of (both totally real and complex) cubic fields $K$ such that the curves $C_K \colon z^3=f_K(x,y)$ all have local solutions but not global solutions.

Choosing $t$ large enough in the proof of Theorem~\ref{S3main} also immediately yields Theorem~\ref{shatheorem}.  More generally, if $\Sigma \subset \Z$ is a set of integers that are cubefree and satisfy a further finite set of congruence conditions, then, for a positive proportion of $k \in \Sigma$, we have $\dim_{\F_3}\Sha(E_k)[3] \geq r$. 
Indeed, for any $m$, the set $T_m \cap \Sigma$ has positive density, and 
$\dim_{\F_3}\Sel_{\phi_k}(E_k)\geq m$ for {\it all} $k \in T_m$, by \cite[Thm.~4(ii)]{j=0}.  On the other hand, by Corollary \ref{cor:averagerankbound}, the average rank of $E_k$ is at most~$3/2$.  Thus, for $m \geq r + 2$, it cannot be the case that, for $100\%$ of $k \in T_m \cap \Sigma$, the image of $\Sel_{\phi_k}(E_k)$ in $\Sha(E_k)[3]$ is less than $r$-dimensional.\qed

\bibliographystyle{abbrv}
\bibliography{references}

\end{document}